\documentclass[onefignum,onetabnum]{siamart190516}

\usepackage{algorithmic}
\usepackage{amsfonts}
\usepackage{amsopn}
\usepackage{epstopdf}
\usepackage{graphicx}

\newsiamremark{remark}{Remark}
\newsiamremark{hypothesis}{Hypothesis}
\crefname{hypothesis}{Hypothesis}{Hypotheses}
\newsiamthm{claim}{Claim}

\ifpdf
  \DeclareGraphicsExtensions{.eps,.pdf,.png,.jpg}
\else
  \DeclareGraphicsExtensions{.eps}
\fi

\usepackage{amsmath}
\usepackage{amsopn}
\usepackage{amssymb}
\usepackage{array}
\usepackage{bm}
\usepackage{booktabs}
\usepackage{caption}
\usepackage{fix-cm}
\usepackage{hyperref}
\usepackage{lineno}
\usepackage[cal=boondoxo]{mathalfa}
\usepackage{mathtools}
\usepackage{microtype}
\usepackage{multirow}
\usepackage{pgf}
\usepackage{ragged2e}
\usepackage[onehalfspacing]{setspace}
\usepackage{subcaption}
\usepackage{tabularx}
\usepackage{tikz}
\usepackage{ulem}
\usepackage{url}
\usepackage{verbatim}

\usetikzlibrary{calc}
\usetikzlibrary{plotmarks}

\newcommand{\fat}[1]{\ifmmode\bm{#1}\else\textbf{#1}\fi}

\newcommand{\set}[1]{\mathbb{#1}}

\newcommand{\vect}[1]{\ensuremath{\fat{#1}}}

\newcommand{\matr}[1]{#1}

\newcommand{\tens}[1]{\mathcal{#1}}

\newcommand{\func}[1]{\textsf{#1}}

\newcommand{\vx}[0]{\vect{x}}            

\newcommand{\mg}[0]{\matr{G}}            

\newcommand{\tg}[0]{\tens{G}}            
\newcommand{\ty}[0]{\tens{Y}}            

\newcommand{\ff}[0]{\func{f}}            
\newcommand{\hff}[0]{\widehat{\ff}}      

\let\div=\relax
\DeclareMathOperator{\div}{div}
\DeclareMathOperator{\vecop}{vec}

\newcommand{\method}[1]{\textsf{#1}}

\title{Black box approximation in the tensor train format initialized by ANOVA decomposition\thanks{
    Submitted to the editors August 04, 2022.
    \funding{The work was supported by the Ministry of Science and Higher Education of the Russian Federation under grant No. 075-10-2021-068.
}}}

\author{
    Andrei Chertkov\thanks{
        Skolkovo Institute of Science and Technology, Moscow, Russia (\email{a.chertkov@skoltech.ru}, \email{g.ryzhakov@skoltech.ru}, \email{i.oseledets@skoltech.ru}).
    }
    \and
    Gleb Ryzhakov\footnotemark[2]
    \and
    Ivan Oseledets\footnotemark[2]}

\headers{Black box approximation in the tensor train format}{A. Chertkov, G. Ryzhakov, and I. Oseledets}

\ifpdf
\hypersetup{
  pdftitle={Black box approximation in the tensor train format},
  pdfauthor={A. Chertkov, G. Ryzhakov, and I. Oseledets}
}
\fi

\begin{document}

\maketitle

\begin{abstract} 
    Surrogate models can reduce computational costs for multivariable functions with an unknown internal structure (black boxes).
    In a discrete formulation, surrogate modeling is equivalent to restoring a multidimensional array (tensor) from a small part of its elements.
    The alternating least squares (ALS) algorithm in the tensor train (TT) format is a widely used approach to effectively solve this problem in the case of non-adaptive tensor recovery from a given training set  (i.e., tensor completion problem).
    TT-ALS allows obtaining a low-parametric representation of the tensor, which is free from the curse of dimensionality and can be used for fast computation of the values at arbitrary tensor indices or efficient implementation of algebra operations with the black box (integration, etc.).
    However, to obtain high accuracy in the presence of restrictions on the size of the train data, a good choice of initial approximation is essential.
    In this work, we construct the ANOVA representation in the TT-format and use it as an initial approximation for the TT-ALS algorithm.
    The performed numerical computations for a number of multidimensional model problems, including the parametric partial differential equation, demonstrate a significant advantage of our approach for the commonly used random initial approximation.
    For all considered model problems we obtained an increase in accuracy by at least an order of magnitude with the same number of requests to the black box.
    The proposed approach is very general and can be applied in a wide class of real-world surrogate modeling and machine learning problems.
\end{abstract}



\section{Introduction}
\label{s:intro}

Many physical and engineering models can be represented as a real function (output of the model), which depends on a multidimensional argument (input of the model) and looks like
\begin{equation}\label{eq:model}
y = \ff(\vx) \in \set{R},
\quad
\vx = [x_1,\, x_2,\, \ldots,\, x_d]^T \in \Omega \subset \set{R}^d.
\end{equation}
Such functions often have a form of a black box (BB), i.e., its internal structure and smoothness properties remain unknown.
The time and/or required resources for one computation of the function $\ff$ may be significant, and it is relevant to train some surrogate model $\func{g}$ (low-parametric approximation) that can be evaluated quickly, but at the same time remains sufficiently close to the original function.
Then such a simple model $\func{g}$ can be used instead of the original BB for faster computations of its outputs from given inputs.
Moreover, the statistical characteristics of the BB may be approximately recovered from this model $\func{g}$ by fast Monte-Carlo sampling or with the usage of the known specific internal structure of the function $\func{g}$.

The model $\func{g}$ may be a decomposition by some basis functions, for example, Chebyshev polynomials~\cite{trefethen2019approximation}, but in many cases, it is more natural to directly discretize the target function \cref{eq:model} on a multidimensional grid
\begin{equation}\label{eq:grid}
\bigl\{
    x_1^{(n_1)}\!, \,
    x_2^{(n_2)}\!, \,
    \ldots, \,
    x_d^{(n_d)}
\bigr\},
\quad
n_k = 1, 2, \ldots, N_k
\; \text{ for } \;
k = 1, 2, \ldots, d,
\end{equation}
and then represent the BB as implicitly specified multidimensional array (tensor)\footnote{
    By tensors we mean multidimensional arrays with a number of dimensions $d$ ($d \geq 1$).
    A two-dimensional tensor ($d = 2$) is a matrix, and when $d = 1$ it is a vector.
    For scalars we use normal font, we denote vectors with bold letters and we use upper case calligraphic letters ($\tens{A}, \tens{B}, \tens{C}, \ldots$) for tensors with $d > 2$.
    The $(n_1, n_2, \ldots, n_d)$th entry of a $d$-dimensional tensor $\tens{A} \in \set{R}^{N_1 \times N_2 \times \ldots \times N_d}$ is denoted by $\tens{A}[n_1, n_2, \ldots, n_d]$, where $n_k = 1, 2, \ldots, N_k$ ($k = 1, 2, \ldots, d$) and $N_k$ is a size of the $k$-th mode, and mode-$k$ slice of such tensor is denoted by $\tens{A}[n_1, \ldots, n_{k-1}, :, n_{k+1}, \ldots, n_d]$.
} $\ty \in \set{R}^{N_1 \times N_2 \times \ldots \times N_d}$ that collects all possible discrete values of the function \cref{eq:model} inside the domain $\Omega$, i.e.,
\begin{equation}\label{eq:tensor-function}
\ty[n_1,\, n_2,\, \ldots,\, n_d]
=
\ff\bigl(
    x_1^{(n_1)}\!, \, x_2^{(n_2)}\!, \, \ldots, \, x_d^{(n_d)}
    \bigr).
\end{equation}

Undoubtedly, the task of explicitly constructing and storing such a tensor is too computationally expensive, and for large values of the dimension $d$, this is completely impossible due to the course of the dimensionality.
However, the usage of the low-rank tensor approximations, namely the tensor train (TT) decomposition~\cite{oseledets2011tensor}, makes it possible to approximately represent the tensor in a compact low-parameter format using only a small number of explicitly computed elements.
The TT-decomposition is a common approach for compact approximation of multidimensional arrays and multivariable functions~\cite{cichocki2016tensor, cichocki2017tensor, phan2020tensor}.
The approximation in the TT-format allows subsequent usage both for quick calculation of BB values and for constructing its various statistical characteristics.
It is possible to effectively perform algebraic operations (element-by-element addition and multiplication, convolution, etc.)
over tensors in the TT-format~\cite{oseledets2011tensor}.
Thus, for example, 
it turns out to be more efficient in some cases 
to construct a surrogate model of a multidimensional tensor in the TT-format first, 
and then perform summation with it (see, e.g.,~\cite{ballester2022tensor, ryzhakov2022constructive}).

The TT-ALS (alternating least squares in the TT-format)~\cite{holtz2012alternating, grasedyck2013alternating} is a tensor completion method that constructs the TT-approximation from a given set of tensor elements (e.g., random samples may be used), alternately optimizing the current approximation for each mode, with fixed values of the parameters corresponding to the rest of the modes.
The TT-ALS is a powerful tool for tensor approximation, and it has found a wide range of practical applications~\cite{wang2017efficient, sedighin2021adaptive, chen2021tensor}. 
To use the TT-ALS, it is necessary to set an initial approximation, which can, for example, be given in the form of a random TT-tensor.
However, the choice of this initial approximation plays an important role in the quality of the resulting representation in the TT-format and often the use of a random TT-tensor turns out to be unsatisfactory~\cite{kapushev2020tensor, richter2021solving, dolgov2019hybrid, chertkov2021solution}.
So, in particular, with an unsuccessful initial approximation, too many iterations of the method may be required, or the resulting solution may converge to a local optimum that is not good enough.
In this work, we propose the modified version of the ANOVA (analysis of variance) representation~\cite{sobol2001global} in the TT-format as an initial approximation for the TT-ALS algorithm.
The interdependence of the model inputs utilizing a split representation corresponding to the ANOVA decomposition allows significantly increase the stability of the TT-ALS approach for constructing a TT-tensor from restricted observations.

As a practically significant numerical example, we consider a parameter-dependent partial differential equation (PDE).
Equations of this kind have a wide range of applications~\cite{cliffe2011multilevel, babuska2004galerkin} such as those associated with optimization or uncertainty quantification after random field discretization.
These include groundwater heights in geotechnical engineering, soil parameters, wind loads and snow loads in structural engineering, the amount of precipitation and evaporation in hydrology, etc.
We consider the mean value of the PDE solution as a multivariable function of PDE parameters and construct its compact representation in the TT-format using the proposed approach.
The performed numerical computations for this problem, as well as for a number of model analytical multivariable functions, demonstrate a significant advantage of our approach concerning the commonly used random initial approximation.

To summarize, our main contributions are the following:
\begin{itemize}
    \item we construct the ANOVA expansion in the form of the TT-tensor (TT-ANOVA) and we propose to use it as an initial approximation for the TT-ALS algorithm to approximate the multidimensional BB (\func{TT-ANOVA-ALS} method);
    \item we implement the proposed algorithms as a publicly available python package\footnote{
        We implemented basic operations in the TT-format, as well as TT-ALS and TT-ANOVA approaches within the framework of the software product \func{teneva}, which is available from \url{https://github.com/AndreiChertkov/teneva}.
    };
    \item we apply our approach\footnote{
        The program code with numerical examples, given in this work, is publicly available in the repository \url{https://github.com/AndreiChertkov/teneva_research_anova_and_als}.
        For all considered model problems, including parametric PDE, we obtained an increase in accuracy by at least an order of magnitude with the same number of requests to the BB.
    } for several multidimensional model problems, including the approximation of the parametric PDE solution, to demonstrate its robustness and performance.
\end{itemize}

\section{Tensor train format}
\label{s:tt}

There has been much interest lately in the development of data-sparse tensor formats for high-dimensional problems.
A very promising tensor format is provided by the TT-approach~\cite{oseledets2011tensor}.
It can be computed via standard decompositions (such as SVD and QR)~\cite{oseledets2009breaking}, but does not suffer from the curse of dimensionality.

A tensor $\ty \in \set{R}^{N_1 \times N_2 \times \ldots \times N_d}$ is said to be in the TT-format, if its elements are represented by the following formula (see also illustration on \cref{fig:tt-element})
\begin{multline}
\label{eq:tt-repr-tns}
\ty [n_1, n_2, \ldots, n_d]
=
\sum_{r_1=1}^{R_1}
\sum_{r_2=1}^{R_2}
\cdots
\sum_{r_{d-1}=1}^{R_{d-1}}
    \tg_1 [1, n_1, r_1]
    \tg_2 [r_1, n_2, r_2]
     \cdots 
    \\
    \tg_{d-1} [r_{d-2}, n_{d-1}, r_{d-1}]
    \tg_d [r_{d-1}, n_d, 1],
\end{multline}
where $n_k = 1, 2, \ldots, N_k$ ($k = 1, 2, \ldots, d$) represent the multi-index, three-dimensional tensors $\tg_k \in \set{R}^{R_{k-1} \times N_k \times R_k}$ are named TT-cores, and integers $R_{0}, R_{1}, \ldots, R_{d}$ (with convention $R_{0} = R_{d} = 1$) are named TT-ranks.
The latter formula can be also rewritten in a more compact form
\begin{equation}\label{eq:tt-repr-mtr}
\ty [n_1, n_2, \ldots, n_d]
=
\matr{G}_1(n_1)
\matr{G}_2(n_2)
\cdots
\matr{G}_d(n_d),
\end{equation}
where $\matr{G}_k(n_k) = \tg_k [:, n_k, :]$ is an $R_{k-1} \times R_k$ matrix for each fixed $n_k$ (since $R_{0} = R_{d} = 1$, the result of matrix multiplications in \cref{eq:tt-repr-mtr} is a scalar).

The benefit of the TT-decomposition is the following.
Storage of the TT-cores $\tg_1, \tg_2, \ldots, \tg_d$ requires less or equal than $d \times \max_{1 \leq k \leq d}{\left(N_k R_k^2\right)}$ memory cells instead of $N = N_1 N_2 \ldots N_d \sim N_0^d$ cells for the uncompressed tensor, where $N_0$ is an average size of the tensor modes, and hence the TT-decomposition is free from the curse of dimensionality if the TT-ranks are bounded.
The detailed description of the TT-format and linear algebra operations in terms of this format is given in works~\cite{oseledets2009breaking, oseledets2011tensor}.

\begin{figure}[t!]
    \centering
    \includegraphics[scale=0.48]{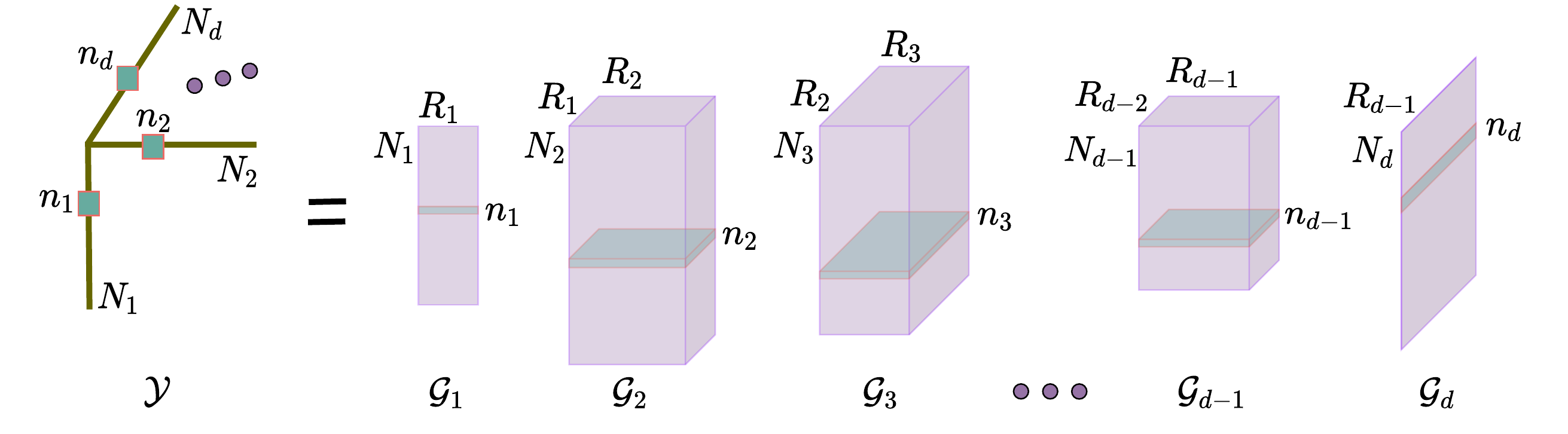}
    \caption{
        Schematic representation of the low-rank TT-decomposition.
        The tensor $\ty \in \set{R}^{N_1 \times N_2 \times \ldots \times N_d}$ is represented in the low-rank TT-format as an ordered list of three-dimensional tensors $\tg_k \in \set{R}^{R_{k-1} \times N_k \times R_k}$ ($k = 1, 2, \ldots, d$), which are named TT-cores.
        To compute an arbitrary element $(n_1, n_2, \ldots, n_d)$ of the tensor $\ty$, the product of the TT-cores according to the formula \cref{eq:tt-repr-tns} or \cref{eq:tt-repr-mtr} should be performed.
    }
    \label{fig:tt-element}
\end{figure}

An exact TT-representation exists for the given full tensor $\widehat{\ty}$, and TT-ranks of such representation are bounded by ranks of the corresponding unfolding matrices~\cite{oseledets2011tensor}.
Nevertheless, in practical applications, it is more useful to construct TT-approximation with a prescribed accuracy $\epsilon_{TT}$, and then carry out all operations (summations, products, etc) in the TT-format, maintaining the same accuracy $\epsilon_{TT}$ of the result.
For a given tensor $\widehat{\ty}$ in the full format and desired accuracy $\epsilon_{TT}$ in the Frobenius norm
\begin{equation}\label{eq:tt-appr-error}
|| \ty - \widehat{\ty} ||_F
\leq
\epsilon_{TT} \cdot || \widehat{\ty} ||_F,
\end{equation}
the TT-decomposition (compression) can be performed by a stable TT-SVD algorithm, but this procedure of the tensor approximation from the full format is too costly and is even impossible for large dimensions since it requires the calculation of all elements of the original tensor.
Therefore, more efficient algorithms, that allow constructing the TT-decomposition from a small part of the tensor elements (i.e., allow to perform tensor completion), are used in multidimensional applications.
We consider below the TT-ALS approach, which is one of the most popular such algorithms.

\section{Black box approximation}
\label{s:scheme}

\begin{figure}[t!]
    \centering
    \includegraphics[scale=0.62]{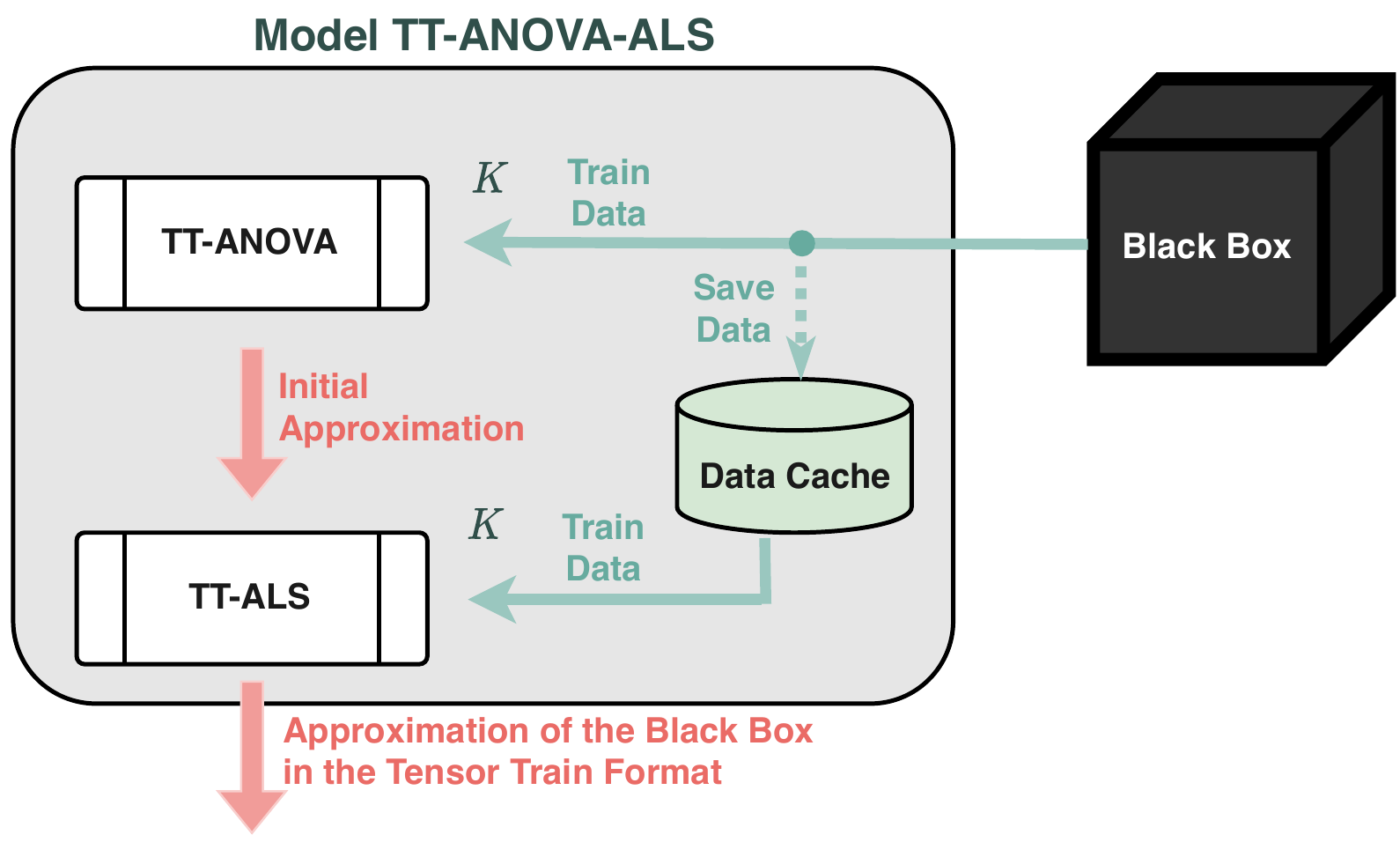}
    \caption{
        The proposed \func{TT-ANOVA-ALS} computation approach for the discretized BB approximation based on a random training dataset.
    }
    \label{fig:scheme}
\end{figure}

The proposed \func{TT-ANOVA-ALS} approach to improve the stability of the BB approximation in terms of the TT-format is schematically shown in \cref{fig:scheme}.
We build the rough approximation for the BB by \func{TT-ANOVA} algorithm (which is presented below in \cref{s:anova}), using $K$ random samples from the BB, and then apply the well-known \func{TT-ALS} algorithm (it is briefly discussed below in \cref{s:als}) on the same $K$ samples using the result of the \func{TT-ANOVA} as an initial approximation.
Note that in this model, there is no need to access the~BB during iterations of the algorithm, just a random training dataset is sufficient.
We show by numerical examples, that this approach significantly increases the final accuracy of the approximation with virtually no increase in the computational complexity and eliminates the possibility of an accidental failure of the \func{TT-ALS} when the random initial approximation is chosen poorly.

\subsection{ANOVA decomposition in the tensor train format}
\label{s:anova}

We assume that our model function \cref{eq:model} is defined in the hypercube $\Omega=[0,\,1]^d$  and integrable.
Consider its ANOVA representation (or decomposition into summands of different dimensions)~\cite{sobol2001global}
\begin{equation}\label{eq:anova}
\ff(x_1, \ldots, x_d)
=
\ff_0 +
\sum_{i=1}^d \ff_i(x_i) +
\sum_{1 \leq i < j \leq d} \ff_{ij}(x_i, \, x_j) +
\ldots +
\ff_{123 \ldots d}(x_1, \, \ldots, \, x_d),
\end{equation}
where for all the functions $\ff_{i_1 i_2 \ldots i_s}$ ($1 \leq i_1 <  i_2 < \ldots < i_s \leq d$, $s = 1, 2, \ldots,d$) the following equality holds
\begin{equation}\label{eq:anova-cond}
\int_{\Omega}
    \ff_{i_1 \ldots i_k \ldots i_s}
        (x_{i_1}, \, \ldots, \, x_{i_k}, \, \ldots, \, x_{i_s})
    \, dx_{i_k} = 0,
    \quad
    \text{ for } \;
     k = 1, \, 2, \, \ldots, \, s.
\end{equation}
It can be shown from the definition of ANOVA \cref{eq:anova,eq:anova-cond} that this decomposition is unique and all the members are orthogonal, hence the zero and first-order terms can be expressed as follows
\begin{equation}\label{eq:anova-members}
\ff_0
=
\int_{\Omega} \ff(\vx) \, d \vx,
\quad
\ff_i(x_i)
=
\int_{\Omega} \ff(\vx) \, \prod_{k \ne i} dx_k
-
\ff_0,
\quad
\text{ for } \;
i = 1, \, 2, \, \ldots, \, d.
\end{equation}

The ANOVA representation \cref{eq:anova} is often used to estimate the sensitivity of a function to input variables or their combinations (Sobol indices) and to understand the relationship of variables or their effect on the model output~\cite{sobol2001global, ballester2019sobol}.
This information may be useful to eliminate uncertainty or simplify the model, i.e., the variables on which the output depends weakly can be fixed (``frozen'').
However, in our case, the attractiveness of ANOVA lies in the possibility of its fast and efficient representation in the form of a TT-tensor using some samples from the objective function $\ff$.

Suppose we have a training dataset $(X^{(train)}, \, Y^{(train)})$
\begin{align*}
X^{(train)} & = \{
    \vx^{(1)}, \vx^{(2)}, \ldots, \,\vx^{(M)}
\}, &
\vx^{(m)} &\in \set{R}^{d},
\\
Y^{(train)} & = \{
    y^{(1)}, y^{(2)}, \ldots, \,y^{(M)}
\}, &
y^{(m)} &= \ff(\vx^{(m)}) \in \set{R},
\end{align*}
obtained from a uniform distribution, or from any low discrepancy sequence (LHS, Sobol, etc.).
Note that for the tensor approximation problem each variable $x_i$ ($i = 1, 2, \ldots, d$) take only a discrete set of values $v^{(i)}=\{v_1^{(i)},\,v_2^{(i)},\,\ldots, v_{N_i}^{(i)}\}$, moreover, the number of these values $N_i$ for different variables may be different\footnote{
    For simplicity, we will further assume that the number of different values present in the dataset is the same as the size of the corresponding tensor modes $N_1, N_2, \ldots, N_d$.
    This can be achieved, for example, by using the LHS distribution.
}.
Then, we introduce the notation
\begin{equation*}
X_{i,j}^{(train)} = \left\{
    \vx \in X^{(train)} \mid x_i = v_j^{(i)}
\right\},
\quad
\text{ for }\;
j = 1,\, 2,\, \ldots,\, N_i, \;\;
i = 1,\, 2,\, \ldots,\, d.
\end{equation*}
In other words, the set $X_{i,j}^{(train)}$ collects those and only those vectors from $X^{(train)}$, for which the $i$-th component coincides with its $j$-th possible value.

Taking into account the introduced notation, we can represent the zero and first-order terms from \cref{eq:anova-members} as
\begin{equation}\label{eq:anova_res_0}
\hff_0 \approx
    \frac1M \sum_{m=1}^M \ff(\vx^{(m)}) =
    \mathrm{mean} \{
        \ff(\vx) \mid \vx \in X^{(train)}
    \},
\end{equation}
\begin{equation}\label{eq:anova_res_1}
\hff_i(v_j^{(i)}) \approx
    \frac1{|X_{i,j}^{(train)}|}
    \sum_{\vx \in X_{i,j}^{(train)}}
        \ff(\vx) - \hff_0 =
    \mathrm{mean} \{\;
        \ff(\vx) \mid \vx \in X_{i,j}^{(train)}
    \;\}
    - \widehat f_0,
\end{equation}
where $j = 1, 2, \ldots, N_i$, $i = 1, 2, \ldots, d$ and $|\cdot|$ denotes the cardinality of a set.

The resulting approximations \cref{eq:anova_res_0} and \cref{eq:anova_res_1} allow to construct the discretized first order ANOVA representation for a given training dataset.
In the following Theorem, we consider a way to explicitly convert this representation to the TT-format.

\begin{theorem}
Consider the TT-decomposition $\ty$ for tensor $\hat{\ty} \in \set{R}^{N_1 \times N_2 \times \ldots \times N_d}$ with the TT-cores $\tens{G}_1 \in \set{R}^{1 \times N_1 \times 2}$, $\tens{G}_i \in \set{R}^{2 \times N_i \times 2}$ (for $i = 2, 3, \ldots, d-1$) and $\tens{G}_d \in \set{R}^{2 \times N_d \times 1}$ such that
\begin{equation}\label{eq:anova-tt-core-first}
\tens{G}_1 [1, j, :] = \begin{pmatrix}
    1 & \hff_1(v_j^{(1)})
\end{pmatrix},
\quad
j = 1, 2, \ldots, N_1,
\end{equation}
\begin{equation}\label{eq:anova-tt-core-middle}
\tens{G}_i [:, j, :] = \begin{pmatrix}
    1 & \hff_i(v_j^{(i)}) \\
    0 & 1
\end{pmatrix},
\quad
j = 1, 2, \ldots, N_i,
\quad
i = 2, 3, \ldots, d-1,
\end{equation}
\begin{equation}\label{eq:anova-tt-core-last}
\tens{G}_d [:, j, 1] = \begin{pmatrix}
    \hff_d(v_j^{(d)}) + \hff_0 \\
    1
\end{pmatrix},
\quad
j = 1, 2, \ldots, N_d,
\end{equation}
where $\hff_i$ (i = 0, 1, \ldots, d) are the zero and first order terms from \cref{eq:anova_res_0} and \cref{eq:anova_res_1}.
Then the TT-tensor $\ty$ has the same values for all possible elements as the corresponding first-order ANOVA decomposition.
\end{theorem}

\begin{proof}
Let directly calculate the matrix product of the TT-cores \cref{eq:anova-tt-core-first}, \cref{eq:anova-tt-core-middle}, \cref{eq:anova-tt-core-last}
\begin{equation*}
\tg_1[1, j_1, :]
\;
\tg_2[:, j_2, :]
=
\begin{pmatrix}
    1 & \hff_1(v_{j_1}^{(1)})
\end{pmatrix}
\begin{pmatrix}
    1 & \hff_2(v_{j_2}^{(2)}) \\
    0 & 1
\end{pmatrix}
=
\begin{pmatrix}
    1 &
        \hff_1(v_{j_1}^{(1)}) +
        \hff_2(v_{j_2}^{(2)})
\end{pmatrix},
\end{equation*}
then
\begin{equation*}
\begin{split}
\tg_1 [1, j_1, :]
\;
\tg_2 [:, j_2, :]
\;
\tg_3 [:, j_3, :]
& =
\begin{pmatrix}
    1 & \hff_1(v_{j_1}^{(1)}) + \hff_2(v_{j_2}^{(2)})
\end{pmatrix}
\begin{pmatrix}
1 & \hff_3(v_{j_3}^{(3)}) \\
0 & 1
\end{pmatrix}
\\
& =
\begin{pmatrix}
    1 &
        \hff_1(v_{j_1}^{(1)}) +
        \hff_2(v_{j_2}^{(2)}) +
        \hff_3(v_{j_3}^{(3)})
\end{pmatrix},
\end{split}
\end{equation*} 
and so on by induction, up to the last TT-core
\begin{equation*}
\begin{split}
&
\tg_1 [1, j_1, :]
\;
\tg_2 [:, j_2, :]
\;
\ldots
\;
\tg_{d-1} [:, j_{d-1}, :]
\;
\tg_{d} [:, j_d, 1]
=
\\
&
\begin{pmatrix}
    1 &
        \hff_1(v_{j_1}^{(1)}) +
        \hff_2(v_{j_2}^{(2)}) +
        \ldots +
        \hff_{d-1}(v_{j_{d-1}}^{(d-1)})
\end{pmatrix}
\begin{pmatrix}
    \hff_d(v_{j_d}^{(d)}) + \hff_0 \\
    1
\end{pmatrix} =
\\
&
\hff_0 +
\hff_1(v_{j_1}^{(1)}) +
\hff_2(v_{j_2}^{(2)}) +
\ldots +
\hff_{d}(v_{j_{d}}^{(d)}),
\end{split}
\end{equation*}
that is, the $(j_1, j_2, \ldots, j_d)$th element of the TT-tensor (taking into account the formula \cref{eq:tt-repr-tns}) coincides with the ANOVA representation \cref{eq:anova}.
\end{proof}

Thus, the given scheme for choosing the elements of the TT-cores allows us to obtain the ANOVA representation within the framework of the low-rank TT-decomposition, and we call this representation \func{TT-ANOVA}.
Note that this result can be directly generalized to the case of arbitrary TT-rank if we take the corresponding additional elements of the TT-cores equal to zero.

\subsection{ALS algorithm for tensor completion}
\label{s:als}

The idea of the TT-ALS tensor completion algorithm~\cite{holtz2012alternating, grasedyck2013alternating} is to successively refine each of the TT-cores by solving an overdetermined system of linear equations with the right-hand side equal to the given tensor values from the train dataset $(X^{(train)}, \, Y^{(train)})$ using the least-squares method.
The TT-ALS starts from some initial approximation, which is usually a random TT-tensor, and successively refines the TT-cores in passes from left to right (i.e., from the first TT-core to the last one) and back.
One complete pass from left to right and back is called a sweep, and the total number of sweeps is a hyperparameter of the algorithm.

\begin{algorithm}[t]
\caption{TT-ALS algorithm}
\label{alg:als}
\begin{algorithmic}[1]
\REQUIRE{
    input multi-indices~$J$ and related outputs~$Y$ of the BB;
    initial approximation $\tg_1,\,\tg_2,\,\ldots,\,\tg_d$ ($\tg_i \in \set{R}^{R_{i-1} \times N_i \times R_i}$ for $i=1,\ldots,d$})
\ENSURE{
    TT-cores of the tensor that approximate the given BB}

\WHILE{stooping criterion is not met}

\FOR{$i$ from $1$ to $d$}
\FOR{$n_k$ from $1$ to $N_k$}
\STATE{$l\gets 1$}
\STATE{Allocate 
$\matr{A} \in \set{R}^{K\times (R_{i-1} R_{i})}$,
$\vect{b} \in \set{R}^{K}$,
where $K$ is the number of samples $J$ for which the 
$i$-th component equal to $n_k$
}
\FOR{$\vect{j}$ from $J$ such that its $i$-th component equal to $n_k$}
\STATE{$\matr{A} [l,:] \gets \vect{a}$, where vector $\vect{a}$ is defined in \cref{eq:als_compact} and is
based on the current TT-cores and the components of the vector $\vect{j}$}
\STATE{$\vect{b} [l] \gets $ the value from $Y$ corresponds to input $\vect{j}$}
\STATE{$l \gets l + 1$}
\ENDFOR
\STATE{Solve $\matr{A} \, \vect{g} = \vect{b}$ for $\vect{g}$ using least squares method}
\STATE{Update $\tg_i[:,n_k,:] \gets \text{reshape}(\vect{g},\, (R_{i-1}, \, R_i))$}
\ENDFOR
\ENDFOR
\ENDWHILE
\RETURN $\tg_1,\,\tg_2,\,\ldots,\,\tg_d$
\end{algorithmic}
\end{algorithm}

Suppose at the current step we need to update the values of the $i$-th TT-core.
Consider a point $\vx$ from the train dataset in which the value is known and is equal to $y = \ff(\vx)$.
In the discrete setting, this point corresponds to some multi-index $\vect{j} = \left( j_1, \, j_2, \, \ldots, \, j_d \right)$ of the tensor.
We denote by $\vect{g}_l$ and $\vect{g}_r$ the two vectors that are obtained after the multiplication of the TT-cores to the left and right of the $i$-th TT-core correspondingly and by $\mg$ the slice of the $i$-th TT-core:
\begin{align*}
\vect{g}_l^T = &
    \;
    \tg_1 [1, \, j_1, \, :]
    \;
    \tg_2 [:, \, j_2, \, :]
    \;
    \cdots
    \;
    \tg_{i-1} [:, \, j_{i-1}, \, :],
&
\vect{g}_l &
\in \set{R}^{R_{i-1} },
\\
\vect{g}_r = &
    \;
    \tg_{i+1} [:, \, j_{i+1}, \, :]
    \;
    \tg_{i+2} [:, \, j_{i+2}, \, :]
    \;
    \cdots
    \;
    \tg_{d} [:, \, j_{d}, \, 1],
&
\vect{g}_r &
\in \set{R}^{R_i },
\\
\mg = &
    \;
    \tg_i [:, \, j_{i}, \, :],
\quad
\mg \in 
\set{R}^{R_{i-1} \times R_{i}}, &&
\end{align*}
where $R_{i-1}$ and $R_{i}$ are the corresponding TT-ranks.
Given the explicit form of the TT-decomposition \cref{eq:tt-repr-tns}, we need to achieve, at least approximately, the equality
\begin{equation}\label{eq:als_equality}
\vect{g}_l^T \, \mg \, \vect{g}_r = y,
\end{equation}
by selecting the elements of the matrix $\mg$.
Note that we can rewrite the left-hand side of the equality \cref{eq:als_equality} in a form of the scalar product
\begin{equation}\label{eq:als_compact}
\begin{split}
& \sum_{r=1}^{R_i \cdot R_{i+1}}
    \! \vect{a}[r] \cdot (\vecop G) [r] = y,
\\
& \vect{a}[r] =
    \vect{g}_l [\, \lfloor (r-1) / R_i \rfloor + 1 \,]
    \, \cdot \,
    \vect{g}_r [\, (r-1)  \mod  R_i + 1 \,],
\end{split}
\end{equation}
where $\vecop \mg \in \set R^{R_{i-1}\cdot R_{i}}$ denotes vectorization of the matrix $\mg$, i.e.,
$$
(\vecop \mg) [\, r_2 + (r_1-1) \cdot R_i \,] =
    \mg [\, r_1, \,r_2 \,],
\quad
r_1 = 1, \, 2, \,\ldots, \, R_{i-1},
\;\;
r_2 = 1, \, 2, \, \ldots, \, R_{i},
$$
operation ``$\cdot\mod\cdot$'' denotes taking the remainder of a division and $\lfloor \cdot \rfloor$ denotes rounding down to the nearest integer.

\begin{table}[t!]
\caption{
    Benchmark functions for performance analysis of the proposed algorithm.}
\label{tbl:benchmarks}
\centering
\begin{small}
\renewcommand{\arraystretch}{2.5}
\begin{tabular}{|p{2.0cm}|p{0.65cm}|p{1.4cm}|p{7.2cm}|}\hline

Function & Ref. & Bounds & Analytical formula
\\ \hline

\func{Ackley}
    & \cite{jamil2013literature}
    & [$-32.768$, $32.768$]
    &
    $
    \ff(\vx)
    =
    - A e^{-B \sqrt{
        \frac{1}{d} \sum_{i=1}^d x_i^2
    }}
    - e^{
        \frac{1}{d} \sum_{i=1}^d \cos{(C x_i)}
    }
    + A
    + e^{1},
    $
    where $A = 20$, $B = 0.2$ and $C = 2 \pi$
    \\ \hline

\func{Alpine}
    & \cite{jamil2013literature}
    & [$-10$, $10$]
    &
    $
    \ff(\vx)
    =
    \sum_{i=1}^{d}
        | x_i \sin{x_i} + 0.1 x_i |
    $
    \\ \hline

\func{Dixon}
    & \cite{jamil2013literature}
    & [$-10$, $10$]
    &
    $
    \ff(\vx)
    =
    (x_1-1)^2 +
    \sum_{i=2}^{d}
        i \cdot \left(
            2 x_i^2 - x_{i-1} \right)^2
    $
    \\ \hline

\func{Exponential}
    & \cite{jamil2013literature}
    & [$-1$, $1$]
    &
    $
    \ff(\vx)
    = - e^{
        - \frac{1}{2}
        \sum_{i=1}^{d} x_i^2
    }
    $
    \\ \hline

\func{Grienwank}
    & \cite{jamil2013literature}
    & [$-600$, $600$]
    &
    $
    \ff(\vx)
    =
    \sum_{i=1}^d \frac{x_i^2}{4000} -
    \prod_{i=1}^d \cos{\left(\frac{x_i}{\sqrt{i}}\right)} + 1
    $
    \\ \hline

\func{Michalewicz}
    & \cite{vanaret2020certified}
    & [$0$, $\pi$]
    &
    $
    \ff(\vx)
    =
    -\sum_{i=1}^d \sin{\left( x_i \right)} \sin^{2m}{\left( \frac{i x_i^2}{\pi} \right)}
    $
    \\ \hline

\func{Piston}
    & \cite{zankin2018gradient}
    & See \cref{tbl:benchmark_piston}
    &
    $
    \ff(M,\, S,\, V_0,\, k,\, P_0,\, T_a,\, T_0)
    =
    2 \pi \sqrt{
        \frac{M}{k + S^2 \frac{P_0 V_0}{T_0} \frac{T_a}{V^2}}
    },
    $
    \newline
    where $
    V = \frac{S}{2k} \left(
        \sqrt{A^2 + 4 k \frac{P_0 V_0}{T_0} T_a} - A
    \right)
    $ and 
    \newline
    $
    A = P_0 S + 19.62 M - \frac{k V_0}{S}
    $
    \\ \hline

\func{Qing}
    & \cite{jamil2013literature}
    & [$0$, $500$]
    &
    $
    \ff(\vx)
    = \sum_{i=1}^d
        \left( x_i^2 - i \right)^2
    $
    \\ \hline

\func{Rastrigin}
    & \cite{dieterich2012empirical}
    & [$-5.12$, $5.12$]
    &
    $
    \ff(\vx)
    =
    A \cdot d + \sum_{i=1}^{d} \left(
        x_i^2 - A \cdot \cos{(2\pi \cdot x_i)}
    \right),
    $
    \newline
    where $A = 10$
    \\ \hline

\func{Rosenbrock}
    & \cite{jamil2013literature}
    & [$-2.048$, $2.048$]
    &
    $
    \ff(\vx)
    =
    \sum_{i=1}^{d-1} \left(
        100 \cdot (x_{i+1} - x_{i}^2)^2 + (1 - x_{i})^2
    \right)
    $
    \\ \hline

\func{Schaffer}
    & \cite{jamil2013literature}
    & [$-100$, $100$]
    &
    $
    \ff(\vx)
    =
    \sum_{i=1}^{d-1} (
        0.5 +
        \frac{
            \sin^2{\left( \sqrt{x_i^2 + x_{i+1}^2} \right)} - 0.5
        }{
            \left( 1 + 0.001 (x_i^2 + x_{i+1}^2)\right)^2
        }
    )
    $
    \\ \hline

\func{Schwefel}
    & \cite{dieterich2012empirical}
    & [$-500$, $500$]
    &
    $
    \ff(\vx)
    =
    418.9829 \cdot d - \sum_{i=1}^d x_i \cdot \sin{(\sqrt{|x_i|})}
    $
    \\ \hline

\end{tabular}
\end{small}
\end{table}

Now we select all points from the training dataset with the value of the $i$-th element of the multi-index equal to the selected value of $j_i$.
Let there be $K$ such points and $K \geq R_{i-1} R_i$.
Thus we have $K$ equations of the form \cref{eq:als_compact} for the unknown~$G$.
We solve this overdetermined system using the method of least squares.
We do so consistently for all values of $j_i = 1,\, 2,\, \ldots,\, N_i$ obtaining all slices of the TT-core $\tg_i$.
And then we repeat this procedure sequentially for all TT-cores $i=1, \, 2,\,\ldots, \, d$, making several such passes from left to right and back.
We summarise the described procedure in \cref{alg:als}.
A stopping condition in the Algorithm may be 
the reaching of a certain number of iterations (sweeps) or a threshold for changing TT-cores components\footnote{
    It is necessary that the training samples contain all possible values of the indices, i.e., $j_i = 1, 2, \ldots, N_i$ for all $i = 1, \, \ldots, \, d$.
    Moreover, we need to have at least $R_{i-1} R_i$ of them for the least-squares matrix to be a full-rank matrix.
    This can be achieved, for example, by using the LHS distribution with a sample of sufficient size.
}.

\begin{table}[t!]
\caption{
    Description of \func{Piston} function parameters. For each parameter, its description, range of values and physical units are provided.}
\label{tbl:benchmark_piston}
\centering
\begin{small}
\renewcommand{\arraystretch}{1.5}
\begin{tabular}{|p{2.05cm}|p{4.2cm}|p{3cm}|p{2cm}|}\hline

\bf Variable & \bf Name & \bf Range & \bf Units
\\ \hline

$M$   & Piston weight           & $[30, 60]$        & $\textit{kg}$
\\ \hline

$S$   & Piston surface area     & $[0.005, 0.020]$  & $\textit{m}^2$
\\ \hline

$V_0$ & Initial gas volume      & $[0.002, 0.010]$  & $\textit{m}^3$
\\ \hline

$k$   & Spring coefficient      & $[1000, 5000]$    & $\textit{N/m}$
\\ \hline

$P_0$ & Atmospheric pressure    & $[90000, 110000]$ & $\textit{N/m}^2$
\\ \hline

$T_a$ & Ambient temperature     & $[290, 296]$      & $\textit{K}$
\\ \hline

$T_0$ & Filling gas temperature & $[340, 360]$      & $\textit{K}$
\\ \hline

\end{tabular}
\end{small}
\end{table}

\section{Numerical experiments}
\label{s:exp}

\begin{figure}[t!]
\begin{center}
\includegraphics[scale=0.35]{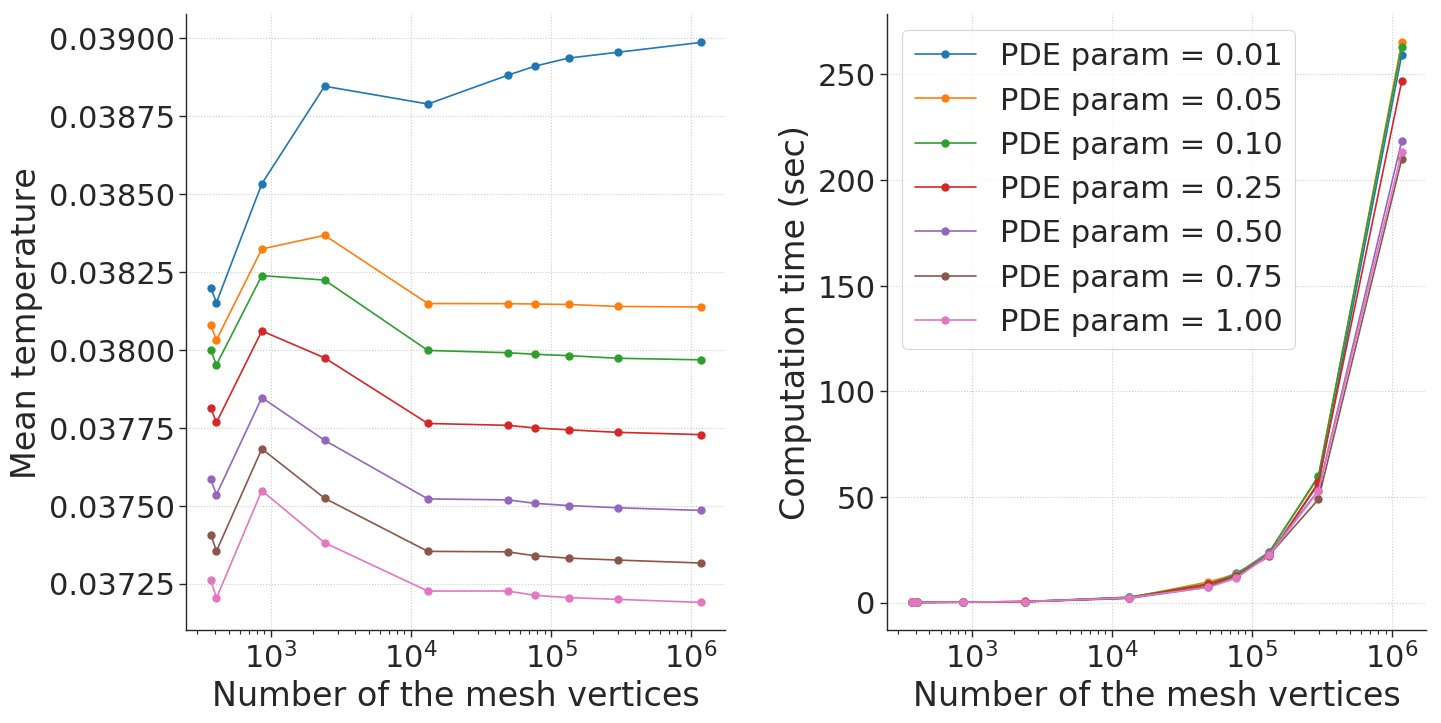}
\caption{
    The \func{PDE-VOI} (on the left plot) and computation time (on the right plot) for different numbers of mesh vertices and different values of the first PDE parameter ($p_1$), while the values of the other parameters ($p_2, \ldots, p_9$) are fixed on the value $0.5$.
}
\label{fig:pde-voi-grid-dep}
\end{center}
\end{figure}

\begin{table}[t!]
\caption{
    Relative error on the train and test data for all benchmarks.
    The reported values for \func{TT-ALS} method (column ``ALS'') are averaged over 10 independent runs.
}
\label{tbl:results}

\begin{center}
\begin{small}
\begin{sc}

\begin{tabular}{|p{2.6cm}|p{1.2cm}|p{2.3cm}|p{2.3cm}|p{2.3cm}|}\hline

          &
          &
ANOVA     & 
ALS       & 
ANOVA-ALS \\ \hline


\multirow{2}{*}{Ackley}
& Train & 1.3e-02 & 5.1e-01 & 2.6e-04 \\ 
& Test  & 1.1e-02 & 7.4e-01 & 2.5e-03 \\ \hline
\multirow{2}{*}{Alpine}
& Train & 2.1e-02 & 8.0e-01 & 1.5e-08 \\ 
& Test  & 2.1e-02 & 1.8e+00 & 1.6e-07 \\ \hline
\multirow{2}{*}{Dixon}
& Train & 4.7e-02 & 6.5e-01 & 3.0e-08 \\ 
& Test  & 4.7e-02 & 2.9e+00 & 1.9e-06 \\ \hline
\multirow{2}{*}{Exponential}
& Train & 1.3e-01 & 6.6e-01 & 5.2e-10 \\ 
& Test  & 1.3e-01 & 2.1e+00 & 1.9e-09 \\ \hline
\multirow{2}{*}{Grienwank}
& Train & 2.0e-02 & 7.1e-01 & 1.1e-04 \\ 
& Test  & 2.0e-02 & 2.0e+00 & 3.9e-04 \\ \hline
\multirow{2}{*}{Michalewicz}
& Train & 3.9e-02 & 6.6e-01 & 2.6e-08 \\ 
& Test  & 4.0e-02 & 5.9e+00 & 2.7e-06 \\ \hline
\multirow{2}{*}{Piston}
& Train & 9.3e-02 & 7.1e-01 & 9.6e-04 \\ 
& Test  & 9.4e-02 & 1.6e+00 & 1.6e-03 \\ \hline
\multirow{2}{*}{Qing}
& Train & 1.5e+01 & 7.4e-01 & 1.2e-08 \\ 
& Test  & 1.4e+01 & 3.4e+00 & 5.7e-05 \\ \hline
\multirow{2}{*}{Rastrigin}
& Train & 7.9e-03 & 8.4e-01 & 2.6e-09 \\ 
& Test  & 8.1e-03 & 1.3e+00 & 2.1e-08 \\ \hline
\multirow{2}{*}{Rosenbrock}
& Train & 2.0e-01 & 7.2e-01 & 1.1e-07 \\ 
& Test  & 2.0e-01 & 4.3e+00 & 2.3e-05 \\ \hline
\multirow{2}{*}{Schaffer}
& Train & 3.9e-02 & 6.8e-01 & 2.4e-04 \\ 
& Test  & 4.0e-02 & 9.9e-01 & 2.7e-04 \\ \hline
\multirow{2}{*}{Schwefel}
& Train & 1.3e-02 & 6.6e-01 & 5.0e-09 \\ 
& Test  & 1.3e-02 & 1.2e+00 & 1.6e-08 \\ \hline
\multirow{2}{*}{PDE-VOI}
& Train & 3.3e-03 & 8.0e-01 & 1.0e-05 \\ 
& Test  & 3.4e-03 & 1.4e+00 & 1.4e-05 \\ \hline


\end{tabular}
\end{sc}
\end{small}
\end{center}
\vskip -0.1in
\end{table}

\begin{table}[t!]
\caption{
    Relative error on the train and test data for all selected benchmarks.
    Multiplicative noise \cref{eq:data_noise} is applied to the training data.
    The reported values for \func{TT-ALS} method (column ``ALS'') are averaged over 10 independent runs.
}
\label{tbl:results_noise}

\begin{center}
\begin{small}
\begin{sc}

\begin{tabular}{|p{2.6cm}|p{1.2cm}|p{2.3cm}|p{2.3cm}|p{2.3cm}|}\hline

          &
          &
ANOVA     & 
ALS       & 
ANOVA-ALS \\ \hline


\multirow{2}{*}{Ackley}
& Train & 1.6e-02 & 6.8e-01 & 8.9e-03 \\ 
& Test  & 1.1e-02 & 9.8e-01 & 7.4e-03 \\ \hline
\multirow{2}{*}{Alpine}
& Train & 2.3e-02 & 8.1e-01 & 8.5e-03 \\ 
& Test  & 2.1e-02 & 1.9e+00 & 1.3e-02 \\ \hline
\multirow{2}{*}{Dixon}
& Train & 4.8e-02 & 7.3e-01 & 8.3e-03 \\ 
& Test  & 4.7e-02 & 3.3e+00 & 1.3e-02 \\ \hline
\multirow{2}{*}{Exponential}
& Train & 1.3e-01 & 7.4e-01 & 7.5e-03 \\ 
& Test  & 1.3e-01 & 2.4e+00 & 2.3e-02 \\ \hline
\multirow{2}{*}{Grienwank}
& Train & 2.2e-02 & 7.1e-01 & 8.2e-03 \\ 
& Test  & 2.0e-02 & 1.8e+00 & 1.5e-02 \\ \hline
\multirow{2}{*}{Michalewicz}
& Train & 4.0e-02 & 6.6e-01 & 7.1e-03 \\ 
& Test  & 3.9e-02 & 6.7e+00 & 3.0e-02 \\ \hline
\multirow{2}{*}{Piston}
& Train & 9.4e-02 & 7.1e-01 & 8.8e-03 \\ 
& Test  & 9.4e-02 & 1.6e+00 & 1.2e-02 \\ \hline
\multirow{2}{*}{Qing}
& Train & 2.7e+01 & 7.4e-01 & 7.9e-03 \\ 
& Test  & 2.6e+01 & 3.4e+00 & 1.8e-02 \\ \hline
\multirow{2}{*}{Rastrigin}
& Train & 1.3e-02 & 6.7e-01 & 8.5e-03 \\ 
& Test  & 8.1e-03 & 1.1e+00 & 7.7e-03 \\ \hline
\multirow{2}{*}{Rosenbrock}
& Train & 2.0e-01 & 7.2e-01 & 8.0e-03 \\ 
& Test  & 2.0e-01 & 4.2e+00 & 2.8e-02 \\ \hline
\multirow{2}{*}{Schaffer}
& Train & 4.0e-02 & 8.5e-01 & 9.0e-03 \\ 
& Test  & 4.0e-02 & 1.2e+00 & 5.9e-03 \\ \hline
\multirow{2}{*}{Schwefel}
& Train & 1.7e-02 & 5.8e-01 & 8.6e-03 \\ 
& Test  & 1.3e-02 & 1.1e+00 & 8.2e-03 \\ \hline


\end{tabular}
\end{sc}
\end{small}
\end{center}
\vskip -0.1in
\end{table}

To demonstrate the effectiveness of the proposed \method{TT-ANOVA-ALS} approach, we consider twelve model $7$-dimensional analytical functions (benchmarks), which are presented in \cref{tbl:benchmarks}. Note that the list of functions includes the Piston function, which corresponds to the practical problem of modeling the time that takes a piston to complete one cycle within a cylinder; the description of related parameters and their ranges are shown in \cref{tbl:benchmark_piston}.

We also consider the more complicated problem of approximating the mean solution of the parameter-dependent PDE~\cite{ballani2015hierarchical, tobler2012low, kapushev2020tensor}
\begin{equation*}
-\div\bigl(
    \func{k} (\vx, \vect{p})
    \nabla \func{u}(\vx, \vect{p})
\bigr) = \ff(\vx),
\quad
\vx = [x_1, x_2]^T \in \Omega = [0, 1]^2,
\quad
\func{u} \big|_{\partial \Omega} = \func{u}_D,
\end{equation*}
where
\begin{equation*}
\ff(\vx) \equiv 1,
\quad
\func{u}_D \equiv 0,
\quad
\func{k}(\vx, \vect{p}) = \begin{cases}
    p_{\mu},
    &
    \text{if} \; \vx \in S_{\mu}, \; \mu = 1,\,2
    ,\ldots,\,m^2,
    \\
    1,
    &
    \text{otherwise},
\end{cases}
\end{equation*}
$\vect{p} = [p_1, p_2, \ldots, p_{m^2}]$ and $\{ S_\mu \}_{\mu=1}^{m^2}$ is a set of $m^2$ disks of radius $\rho = \frac1{4m+2}$ (note that these disks form $m \times m$ grid) with the centers located in the points~$(x^\mu_1,\,x^\mu_2)$,
\begin{multline*}
x^\mu_1 = i \cdot q + (2i-1) \cdot \rho,
\quad
x^\mu_2 = j \cdot q + (2j-1) \cdot \rho,
\\
q = \frac{1 - 2 m \rho}{m + 1},
\quad
\mu=(i-1)m+j,
\quad
i, j = 1, 2, \ldots, m.
\end{multline*}
As a scalar value of interest we consider the average temperature over domain $\Omega$
\begin{equation*}
\func{PDE-VOI} \, (\vect{p})
=
\int_{\Omega} \func{u}(\vx, \vect{p}) \, d \vx,
\quad
\vect{p} \in [0.01, 1]^{m^2}.
\end{equation*}
We select $m = 3$, hence we have $9$ independent parameters, and in discrete representation, the \func{PDE-VOI} is the $9$-dimensional implicit tensor.

For the numerical solution of the PDE, we use a popular software product FEniCS~\cite{logg2012automated} with efficient implementation of the finite element method.
In \cref{fig:pde-voi-grid-dep} we present the \func{PDE-VOI} and computation time for different numbers of vertices of the spatial two-dimensional mesh. As can be estimated from the plots, for sufficient stability and accuracy of the solution, at least $5 \cdot 10^4$ vertices are required (and we will use this number in our computations), while the time of one calculation is more than $5$ seconds. Thus, for this problem, it seems justified to search for a low-parameter approximation, which makes it possible to carry out calculations an orders of magnitude faster.

For all benchmarks, including \func{PDE-VOI}, we consider the tensor that arises when the corresponding function is discretized on a uniform grid with $10$ nodes in each dimension.
We fixed the TT-rank value at $5$ and the number of ALS sweeps at $50$.
The calculations with a random initial approximation were repeated $10$ times and the averaged result was used.
To train the model, we used $10^4$ tensor elements, which are randomly generated from the LHS distribution.
To check the accuracy of the obtained approximation, we used a set of another $10^4$ random samples and calculated the mean square error on it.

The results of numerical calculations are presented in \cref{tbl:results}.
We report the error on the train and test data for each benchmark and the following approximation methods: the \func{TT-ANOVA} algorithm, i.e., the 1st order ANOVA in the TT-format (column ``ANOVA''); the \func{TT-ALS} algorithm with random initial approximation\footnote{
    We generate all TT-cores $\tg_k$ ($k = 1, 2, \ldots, d$) from a random normal distribution.
} (column ``ALS''); the \func{TT-ANOVA-ALS} algorithm, i.e., the \func{TT-ALS} algorithm, which uses the result of the \func{TT-ANOVA} as an initial approximation (column ``ANOVA-ALS'').
As can be seen from the results, in the latter case (\func{TT-ANOVA-ALS}), the accuracy improves by at least an order of magnitude concerning the \func{TT-ALS} method.

To test the robustness of the algorithm, we repeated the same calculations on the noisy training data, i.e., we replace each train value $y$ by
\begin{equation}\label{eq:data_noise}
    \hat{y} = \left(
        1 + 10^{-2} z
    \right) \cdot y,
    \quad
    z \sim \mathcal{N}(0,\, 1).
\end{equation}
The corresponding results are reported in \cref{tbl:results_noise}.
As follows, even in the presence of noise, a high approximation accuracy is maintained when using the \func{TT-ANOVA-ALS} approach.

Thus, the proposed \func{TT-ANOVA-ALS} approach significantly improves the accuracy of the \func{TT-ALS} method 
with virtually no increase in the computational complexity\footnote{
    Calculations were conducted on a regular laptop.
    For our model problem, the time for building the \func{TT-ANOVA} approximation turns out to be less than $10$ milliseconds, and the average time for building the \func{TT-ALS} approximation is about $6$ seconds.
}.
Also note that this approach is deterministic (on a fixed training set), which means that it is not prone to random failures of the \func{TT-ALS} when the initial approximation is chosen poorly.

\section{Related work}
\label{s:related}

Today, TT-decomposition is a popular approach for compact approximation of multidimensional arrays and multivariable functions~\cite{cichocki2016tensor, phan2020tensor, sedighin2021image, sedighin2021adaptive}, including applications in the field of data analysis, machine learning, solution of PDEs, etc.
Various practical methods based on the TT-decomposition have been proposed to solve multidimensional~\cite{richter2021solving, chertkov2021solution}, parametric~\cite{dolgov2019hybrid, glau2020low} and multiscale~\cite{oseledets2016robust, chertkov2016robust} PDEs, to optimize multivariable functions~\cite{sozykin2022ttopt, nikitin2022quantum}, etc.
We note that in the works~\cite{richter2021solving, chertkov2021solution, dolgov2019hybrid, glau2020low}, special attention is paid to the choice of the initial approximation for constructing the TT-decomposition for the PDE solution, and as such an approximation, the solution obtained from the previous time step is used.
At the same time, in the works~\cite{oseledets2016robust, chertkov2016robust, sozykin2022ttopt, nikitin2022quantum}, a random initial approximation is used.

Until now only a few works have been directly devoted to the selection of the initial approximation in the TT-format.
We note the recent work~\cite{kapushev2020tensor}, where the authors apply Gaussian process regression to construct a preliminary approximation of the function, using a training dataset.
After that, the initial approximation is generated by the TT-cross method, which is applied to the regression model.
This approach is very promising, but it
requires significant additional numerical computations to run TT-cross and it also has all the limitations of regression methods when applied to multidimensional problems.
We also note that the authors conducted numerical experiments for only one benchmark, which does not allow us to evaluate the general robustness of the approach.

An interesting method was considered in~\cite{ko2020fast} for constructing an initial approximation in the framework of the proposed new completion method.
The authors perform a special interpolation based on the known tensor values, and then build an initial approximation using the classical TT-SVD algorithm.
However, this approach was developed only for the problem of image and video processing, and it cannot be transferred to essentially multidimensional problems due to the need to build a full tensor for the TT-SVD algorithm.
Also, we note that the ANOVA decomposition was already used in combination with TT-approach, see, for example~\cite{zhang2014enabling, ballester2019sobol}.
But in these works, effective methods for construction of the Sobol indices through the calculation of multidimensional integrals by TT-decomposition were considered, and this is a different problem than the one presented in our work.

\section{Conclusions}
\label{s:concl}

In this work, we proposed the \func{TT-ANOVA} representation as an initial approximation for the \func{TT-ALS} approach, which makes it possible to effectively approximate functionally given 
black box.
We have implemented the corresponding method \func{TT-ANOVA-ALS} in software and demonstrated its effectiveness for the list of model multidimensional problems, including the approximation of the solution to the parametric PDE.
For all considered model problems we obtained an increase in accuracy by at least an order of magnitude with the same number of train samples from the black box and with virtually no increase in the computational complexity.
The proposed approach is deterministic (on a fixed training set) and it is not prone to random failures of the \func{TT-ALS} when the initial approximation is chosen poorly.
The \func{TT-ANOVA-ALS} method can be applied to a wide class of practically significant problems, including surrogate modeling and various machine learning applications.

\bibliographystyle{siamplain}
\bibliography{biblio}

\end{document}